\documentclass[letterpaper]{article}

\usepackage{amsmath,amsthm,bbm,amssymb,amsfonts,latexsym,mathrsfs}
\usepackage[small,nohug,heads=vee]{diagrams}
\diagramstyle[labelstyle=\scriptstyle]

\RequirePackage[left=1in,right=1in,top=1in,bottom=1in]{geometry}

\newcommand{\Z}{\mathbb{Z}}

\newcommand{\Q}{\mathbb{Q}}

\newcommand{\C}{\mathbb{C}}

\newtheorem{theorem}{Theorem}[section]
\newtheorem{lemma}[theorem]{Lemma}

\newtheorem{corollary}{Corollary}[theorem]
\newtheorem{proposition}[theorem]{Proposition}
\theoremstyle{definition}
\newtheorem{definition}{Definition}[section]

\DeclareMathOperator{\GL}{GL}
\DeclareMathOperator{\Rep}{Rep}
\DeclareMathOperator{\im}{im}
\DeclareMathOperator{\Sh}{Sh}
\DeclareMathOperator{\Hom}{Hom}

\DeclareMathOperator{\Fun}{Fun}
\DeclareMathOperator{\vect}{vect}

\DeclareMathOperator{\End}{End}
\DeclareMathOperator{\Ind}{Ind}
\DeclareMathOperator{\cInd}{c-Ind}

\DeclareMathOperator{\Supp}{Supp}
\DeclareMathOperator{\Res}{Res}

\DeclareMathOperator{\Ann}{Ann}

\begin{document}

\title{The Bernstein Center of a $p$-adic Unipotent Group}
\author{Justin Campbell \\ University of Michigan, Ann Arbor}
\date{16 March 2011}
\maketitle

\begin{abstract}

Fran\c cois Rodier proved that it is possible to view smooth representations of certain totally disconnected abelian groups (the underlying additive group of a finite-dimensional $p$-adic vector space, for example) as sheaves on the Pontryagin dual group. For nonabelian totally disconnected groups, the appropriate dual space necessarily includes representations which are not one-dimensional, and does not carry a group structure. The general definition of the topology on the dual space is technically unwieldy, so we provide three different characterizations of this topology for a large class of totally disconnected groups (which includes, for example, $p$-adic unipotent groups), each with a somewhat different flavor. We then use these results to demonstrate some formal similarities between smooth representations and sheaves on the dual space, including a concrete description of the Bernstein center of the category of smooth representations.

\end{abstract}

\section{Introduction}

\label{intro}

This section constitutes a survey of our main results: proofs are omitted for ease of reading, as are technical points which are not critical to the main storyline. The remaining sections, with the exception of the appendix, follow the ordering of the introduction with proofs and the necessary details included.

The main goal of this paper is to prove the following theorem. Let $G$ be a $p$-adic unipotent group and consider the Bernstein center of the category of smooth representations of $G$, which we denote by $\mathscr{Z}(G)$ (all of the relevant definitions are given below). Write $\widehat{G}$ for the space of isomorphism classes of irreducible smooth representations of $G$, equipped with the Fell topology, and $C^{\infty}(\widehat{G}) \subset \Fun(\widehat{G})$ for the $\C$-algebras of locally constant and arbitrary $\C$-valued functions on $\widehat{G}$. A suitable version of Schur's lemma tells us that there is a natural algebra homomorphism $\mathscr{Z}(G) \to \Fun(\widehat{G})$, and the results of section \ref{stalkstatements} on ``stalks" of representations may be used to show that this homomorphism is injective. We shall prove that the image of $\mathscr{Z}(G)$ is precisely $C^{\infty}(\widehat{G})$, which is stated in a slightly more general form as Theorem \ref{bernstein} below.

\subsection{Acknowledgements}

First and foremost, I would like to thank Mitya Boyarchenko for his guidance throughout the writing of this paper. The main results were conjectured by him, and in many cases the technical results and the key ideas of proofs are due to him as well. I also express my gratitude to the NSF for funding my summer REU project which eventually gave rise to this paper, and the University of Michigan Mathematics Department for their support of undergraduate research. I would especially like to thank Stephen DeBacker for his encouragement.

\subsection{Basic notions and results}

We aim to study the representation theory of $p$-adic unipotent groups, but for most of this paper we work with totally disconnected topological groups under a few other hypotheses, introduced below, so as not to limit the applicability of our results. First we introduce some convenient terminology.

\begin{definition}

By an $\ell$-\emph{space} we shall mean a Hausdorff, locally compact, and totally disconnected topological space. An $\ell$-\emph{group} is a topological group whose underlying space is an $\ell$-space.

\end{definition}

Recall the following technically indispensable characterization of $\ell$-groups: a Hausdorff topological group $G$ is an $\ell$-group if and only if the identity element of $G$ has a basis of neighborhoods consisting of compact open subgroups. Clearly discrete groups are $\ell$-groups, and compact $\ell$-groups are just profinite groups.

\begin{definition}

Let $G$ be an $\ell$-group and $(\pi,V)$ a complex representation of $G$, meaning a complex vector space $V$ and a homomorphism $\pi : G \to GL(V)$. We shall say that $(\pi,V)$ is \emph{smooth} provided that the action map $G \times V \to V$ is continuous when $V$ is given the discrete topology. If, in addition, the subspace of $V$ fixed by any compact open subgroup of $G$ is finite-dimensional, we call $(\pi,V)$ \emph{admissible}.

\end{definition}

Note that a complex representation $(\pi,V)$ of $G$ is smooth if and only if the map $g \mapsto \pi(g)(v)$ is locally constant for every $v \in V$, which is to say that any vector in $V$ has an open stabilizer in $G$. In the sequel we may abuse terminology by referring to smooth representations simply as representations, since we are only concerned with smooth representations in this paper. We write $\mathcal{R}(G)$ for the category of smooth representations of an $\ell$-group $G$, a full abelian subcategory of the category $\Rep(G)$ of all complex representations of $G$. Recall that if $K$ is a compact $\ell$-group, then the category $\mathcal{R}(K)$ is semisimple.

Given any complex representation $(\rho,W)$ of $G$ (not necessarily smooth), we write $W^{sm}$ for the smooth subrepresentation of $W$ consisting of vectors whose stabilizer is open in $G$, called the \emph{smooth part of} $W$. It is not hard to see that $W \mapsto W^{sm}$ extends to a functor $\Rep(G) \to \mathcal{R}(G)$ which is right adjoint to the inclusion functor $\mathcal{R}(G) \to \Rep(G)$. Now if $(\pi,V)$ is a smooth representation of $G$, the usual dual representation $V^*$ is not necessarily smooth, so we define the \emph{smooth dual of} $(\pi,V)$ to be $V^{\vee} = (V^*)^{sm}$. The usefulness of the admissibility condition is partially explained by the following easy result.

\begin{proposition}

A smooth representation $(\pi,V)$ is admissible if and only if the natural monomorphism $V \to (V^{\vee})^{\vee}$ is an isomorphism.

\label{admissible}

\end{proposition}

The additive group $(\Q_p,+)$ is an important example of an $\ell$-group. The multiplicative group $(\Q_p \setminus \{ 0 \}, \cdot)$ is also an $\ell$-group, but lacks a useful property: it doesn't have enough compact open subgroups ``at the top."

\begin{definition}

An $\ell_c$-\emph{group} is an $\ell$-group which is a filtered union of its compact open subgroups. This means that every element of an $\ell_c$-group is contained in a compact open subgroup, and that any two compact open subgroups are contained in a third.

\end{definition}

For example, $\Z_p \subset p^{-1}\Z_p \subset p^{-2}\Z_p \subset \cdots$ is a filtration of $(\Q_p,+)$ by compact open subgroups. By taking (topological group) products it is easy to see that the additive group of any finite-dimensional $p$-adic vector space $V$ is an $\ell_c$-group. Recall that the Pontryagin dual of $V$ is noncanonically isomorphic to $V$ as a topological group, and more generally, the Pontryagin dual of an abelian $\ell_c$-group is totally disconnected (therefore an $\ell$-group). The next theorem tells us that smooth representations of such a group can be viewed as sheaves of complex vector spaces on its Pontryagin dual in a natural way.

\begin{theorem}[Rodier]

Let $G$ be an abelian $\ell_c$-group. Then the category $\mathcal{R}(G)$ is equivalent to the category $\Sh(\widehat{G})$ of sheaves of complex vector spaces on the Pontryagin dual of $G$.

\label{rodier}

\end{theorem}

Fran\c cois Rodier proved this remarkable theorem in \cite{Rod}. Notice that the category $\Sh(\widehat{G})$ depends only on the underlying topological space of $\widehat{G}$. The appendix to this paper covers some of the same material and includes an alternate proof of this result (in English).

\subsection{$p$-adic unipotent groups}

A more interesting (in particular, nonabelian) class of $\ell_c$-groups are the $p$-\emph{adic unipotent groups}, which are constructed as follows: let $\mathfrak{g}$ be a finite-dimensional nilpotent Lie algebra over $\Q_p$. The corresponding $p$-adic unipotent group $G$ has the same underlying topological space as $\mathfrak{g}$, with group operation given by the Campbell-Hausdorff formula. More explicitly: for psychological reasons let $\exp : \mathfrak{g} \to G$ and $\log : G \to \mathfrak{g}$ denote the identity map, and then the group operation is \[ \log(\exp(x) \cdot \exp(y)) = x + y + \frac12 [x,y] + \frac{1}{12} [x,[x,y]] - \frac{1}{12} [y,[x,y]] - \cdots \ \ \ \ x,y \in \mathfrak{g}. \] There is no convergence issue because $\mathfrak{g}$ is nilpotent.

More abstractly, a unipotent algebraic group $\mathbb{G}$ over $\Q_p$ gives rise to a $p$-adic unipotent group $G = \mathbb{G}(\Q_p)$, the group of $\Q_p$-points endowed with the $p$-adic topology. Then we can recover $\mathfrak{g}$ in the usual way as the tangent space of $\mathbb{G}$ at the identity under the Lie bracket.

The simplest class of nonabelian $p$-adic unipotent groups are the $p$-\emph{adic Heisenberg groups}. Let $V$ be a finite-dimensional $p$-adic vector space equipped with a symplectic form \[ \omega : V \times V \to \Q_p, \] so that in particular $V$ has even dimension. As a topological space, $G = V \times \Q_p$. The group operation on $G$ is given by the formula \[ (v,a) \cdot (w,b) = (v+w,a+b+\tfrac12 \omega(v,w)). \]

If we let $Z$ denote the center of $G$, which is isomorphic to $\Q_p$ as a topological group, it is not hard to see that $G$ is a central extension of the additive group of $V$ by $Z$: \[ 1 \to Z \to G \to V \to 1. \]

All the irreducible smooth representations of an abelian $\ell_c$-group are one-dimensional, which we prove in the appendix to this paper, so in particular they are admissible. In fact, it is proved in \cite{Rod} that $p$-adic unipotent groups also have this property, which is so useful that we will give it a name for our convenience.

\begin{definition}

We call an $\ell$-group \emph{tame} if all of its irreducible smooth representations are admissible.

\end{definition}

The proof that abelian $\ell_c$-groups are tame is outlined in the appendix to this paper, and Rodier proved that $p$-adic unipotent groups are tame. It should be noted that tameness implies Schur's lemma by an argument in \cite{Pra}.

\begin{proposition}

If $G$ is a tame $\ell$-group and $\pi$ is an irreducible smooth representation of $G$, then the only linear endomorphisms of $\pi$ which commute with the action of $G$ are the scalars.

\label{schur}

\end{proposition}

\subsection{The dual space and its Fell topology}

It is natural to ask if Theorem \ref{rodier} has some sort of generalization to $p$-adic unipotent groups or, more generally, tame $\ell_c$-groups. The first step is to find the appropriate dual space.

\begin{definition}

The \emph{smooth dual space} $\widehat{G}$ of an $\ell$-group $G$ is the space of isomorphism classes of irreducible smooth representations of $G$. 

\end{definition}

For abelian $\ell_c$-groups, $\widehat{G}$ coincides with the Pontryagin dual and in particular carries a natural group structure. Recall that the additive group $V$ of a finite-dimensional $p$-adic vector space can be noncanonically identified with $\widehat{V}$ as a topological group. Even for nonabelian groups, there is a canonical topology on $\widehat{G}$ which generalizes the ``compact-open" topology usually employed in Pontryagin duality, called the \emph{Fell topology}. The Fell topology is unpleasant to describe in general, but the interested reader can refer to the beginning of Section \ref{topology}. This brings us to our first original theorem: a more understandable characterization of the Fell topology in the cases we are interested in.

\begin{theorem}

Let $G$ be a tame $\ell_c$-group. For any compact open subgroup $K \subset G$ and $\rho \in \widehat{K}$, put \[ \mathscr{V}(K,\rho) = \{ \pi \in \widehat{G} \ | \ \Hom_K(\pi|_K,\rho) \neq 0 \}. \] Then the $\mathscr{V}$'s are open and form a basis for the Fell topology on $\widehat{G}$, meaning every open set in $\widehat{G}$ can be written as a union of $\mathscr{V}$'s.

\label{Fell}

\end{theorem}

One reason we study the representation theory of a $p$-adic unipotent group $G$ is that then $\widehat{G}$ admits a geometric description via the orbit method. Again we fix a nilpotent $p$-adic Lie algebra $\mathfrak{g}$ and let $G$ denote the corresponding $p$-adic unipotent group. Recall that $G$ acts by Lie algebra automorphisms on $\mathfrak{g}$ via conjugation (the adjoint action), hence also acts linearly on the dual vector space $\mathfrak{g}^*$ (the coadjoint action). The orbit method gives a homeomorphism \[ \mathfrak{g}^*/G \longrightarrow \widehat{G}. \] This is even more helpful when combined with the algebro-geometric result that the orbits in an affine variety of an algebraic action by a unipotent group are closed.

When $G$ is the Heisenberg group, the orbits of $G$ in $\mathfrak{g}^*$ are as follows. Write $\mathfrak{g}^*_0 \subset \mathfrak{g}^*$ for the subspace consisting of those $f \in \mathfrak{g}^*$ which vanish on the center of $\mathfrak{g}$: then every point of $\mathfrak{g}^*_0$ is fixed by $G$, and each nonzero translate (i.e. coset) of $\mathfrak{g}^*_0$ is a $G$-orbit. Under the orbit method homeomorphism $\mathfrak{g}^*_0$ corresponds to those elements of $\widehat{G}$ which are trivial on the center $Z$ of $G$, and we can identify this (topological) subspace with $\widehat{V}$, where $V$ denotes the additive group of the symplectic vector space used to construct $G$. The nonzero translates of $\mathfrak{g}^*_0$ correspond to the representations in $\widehat{G}$ which are nontrivial on $Z$, and these may be identified with $\widehat{Z} \setminus \{ 0 \}$ by a suitable version of the Stone-von Neumann theorem. So as a set $\widehat{G}$ is a disjoint union of $\widehat{V}$ and $\widehat{Z} \setminus \{ 0 \}$, endowed with a rather strange topology. In particular, if we choose a sequence of points in $\widehat{Z} \setminus \{ 0 \}$ which converges to $0$ in $\widehat{Z}$, the limit of this sequence in $\widehat{G}$ is all of $\widehat{V}$. The topology on $\widehat{G}$ is locally compact and totally disconnected, but plainly not Hausdorff.

\subsection{Stalks in the category $\mathcal{R}(G)$}

\label{stalkstatements}

For now, $G$ is a tame $\ell_c$-group, such as a $p$-adic unipotent group. It seems that $\mathcal{R}(G)$ shares many properties with $\Sh(\widehat{G})$, the category of sheaves of complex vector spaces on $\widehat{G}$, including the existence of a ``stalk" functor for each $\pi \in \widehat{G}$ that naturally returns the muliplicity (appropriately defined) of $\pi$ in smooth representations of $G$. But first, there is a meaningful notion of support for smooth representations.

\begin{definition}

The \emph{support} of a smooth representation $M \in \mathcal{R}(G)$ is defined by \[ \Supp M = \{ \pi \in \widehat{G} \ | \ \Hom_G(M,\pi) \neq 0 \}. \]

\end{definition}

For example, if $S \subset \widehat{G}$ is any subset of the dual space, then $\Supp \oplus_{\pi \in S} \pi = S$, but in general irreducible representations may be combined in more complicated ways than direct sum. Let's write $\mathcal{R}(G)_S$ for the full subcategory of smooth representations of $G$ with support contained in $S$. If $S = \{ \pi \}$ is a singleton we'll abuse notation and denote $\mathcal{R}(G)_S$ by $\mathcal{R}(G)_{\pi}$.

The next result gives a simple description of $\mathcal{R}(G)_{\pi}$ via an analogue of the skyscraper sheaf construction: let $\C-\mathbf{vect}$ denote the category of complex vector spaces.

\begin{theorem}

Fix a point $\pi \in \widehat{G}$. The functor $\C-\mathbf{vect} \longrightarrow \mathcal{R}(G)$ which sends $V \mapsto V_G \otimes \pi$, where $V_G$ denotes the trivial representation of $G$ on $V$, is an equivalence onto $\mathcal{R}(G)_{\pi}$.

\label{skyscraper}

\end{theorem}

In view of the analogy with sheaves it is not surprising then that this functor has a left adjoint, which sends a smooth representation to its ``stalk" at $\pi$. This can be seen by factoring the skyscraper functor $\C-\mathbf{vect} \to \mathcal{R}(G)$ through the inclusion $\mathcal{R}(G)_{\pi} \to \mathcal{R}(G)$, which has a left adjoint described as follows: for any $M \in \mathcal{R}(G)$, write $M_{\pi}$ for the largest quotient of $M$ supported at $\pi$. Explicitly, put \[ M_{\pi} = M / \bigcap_{f \in \Hom_G(M,\pi)} \ker f. \] The content of the claim that $M \mapsto M_{\pi}$ is left adjoint to the inclusion $\mathcal{R}(G)_{\pi} \to \mathcal{R}(G)_{\pi}$ is the following.

\begin{lemma}

For any $M \in \mathcal{R}(G)$ and $\pi \in \widehat{G}$ we have $M_{\pi} \in \mathcal{R}(G)_{\pi}$.

\label{point}

\end{lemma}

Now we define $M(\pi) = \Hom_G(\pi,M_{\pi})$, the vector space we want to think of as the stalk of $M$ at $\pi$.

These stalk functors have some useful properties in common with sheaf-theoretic stalk functors. For example, $M = 0$ if and only if $M(\pi) = 0$ for every $\pi \in \widehat{G}$: this follows from the fact that $M = 0$ if and only if $\Supp M = \varnothing$. The nontrivial content of this statement is that every nonzero smooth representation has an irreducible quotient, which is to say a maximal subrepresentation. Much more generally, the following result allows us to adapt the proofs of many statements about sheaves to the setting of smooth representations.

\begin{lemma}

A sequence $M \to N \to P$ in $\mathcal{R}(G)$ is exact if and only if the associated sequence of vector spaces $M(\pi) \to N(\pi) \to P(\pi)$ is exact for every $\pi \in \widehat{G}$.

\label{exact}

\end{lemma}

So we can test exactness on stalks. In particular, a morphism of smooth representations is a monomorphism (respectively epimorphism, isomorphism) if and only if the associated map of stalks is a monomorphism (respectively epimorphism, isomorphism) at every point in the dual space.

\subsection{The second characterization of the Fell topology}

Now we state our next characterization of the Fell topology, which has a categorical flavor, but first let us fix some notation. Consider subsets $Z \subset W \subset \widehat{G}$ and write $i_! : \mathcal{R}(G)_Z \to \mathcal{R}(G)_W$ for the inclusion functor.

This notation suggests an analogy with the pushforward functor for sheaves, and indeed our situation is very similar on a formal level. We might ask for a left adjoint to $i_!$, and our wish is granted as long as $Z$ is closed. But in fact something even better is true.

\begin{theorem}

If $W \subset \widehat{G}$ is locally closed and $Z \subset W$, then the inclusion functor $i_! : \mathcal{R}(G)_Z \to \mathcal{R}(G)_W$ admits a left adjoint $i^* : \mathcal{R}(G)_W \to \mathcal{R}(G)_Z$ if and only if $Z$ is closed in $W$. When $i^*$ exists it is exact.

\label{fell2}

\end{theorem}

Notice that $i^*$ exists when $Z$ is a point and $W = \widehat{G}$ by Lemma \ref{point}, which yields the following corollary. 

\begin{corollary}

For any tame $\ell_c$-group $G$, the Fell topology on $\widehat{G}$ is $T_1$.

\label{T1}

\end{corollary}

Theorem \ref{fell2} has another interesting consequence.

\begin{corollary}

An irreducible representation $\pi \in \widehat{G}$ is projective in $\mathcal{R}(G)$ if and only if $\pi$ is an open point of $\widehat{G}$ in the Fell topology.

\label{projective}

\end{corollary}

The proof of Corollary \ref{projective} may not be entirely obvious and is included in Section \ref{adjointproofs}.

\subsection{An interlude on Hecke algebras}

Before we state our last characterization of the Fell topology on $\widehat{G}$ for $G$ a tame $\ell_c$-group, which says that locally the dual space is homeomorphic to the primitive spectrum of a certain unital Hecke algebra equipped with the Jacobson (sometimes called hull-kernel) topology, we need a couple of definitions.

\begin{definition}

Write $\mathcal{H}(G)$ for the \emph{Hecke algebra of} $G$, that is, the convolution algebra of locally constant $\C$-valued functions on $G$ with compact support, which is unital only when $G$ is discrete.

\end{definition}

Recall that the category $\mathcal{R}(G)$ is isomorphic to the full subcategory of $\mathcal{H}(G)$-modules $M$ which are non-degenerate in the sense that $\mathcal{H}(G) \cdot M = M$.

\begin{definition}

Fix a compact open subgroup $K \subset G$ and $\rho \in \widehat{K}$. We write $\mathcal{H}(G,\rho)$ for the convolution algebra of locally constant functions $f : G \to \End_{\C}(W^{\vee})$ with compact support satisfying \[ f(k_1gk_2) = \rho^{\vee}(k_1) \circ f(g) \circ \rho^{\vee}(k_2) \ \ \ \ k_1,k_2 \in K, \ g \in G. \] 

\end{definition}

It is well-known that $\mathcal{H}(G,\rho)$ is canonically isomorphic to $\End_G(\cInd_K^G \rho)$ as a $\C$-algebra, so that for any smooth representation $M \in \mathcal{R}(G)$ the space $\Hom_K(\rho,M|_K)$ of $\rho$-invariants has a natural $\mathcal{H}(G,\rho)$-module structure. We point the reader to \cite{BK} for a detailed discussion, with proofs, of these Hecke algebras. The relevant definitions, notations, and facts about induced representations and Frobenius reciprocity may be found in Section \ref{heckeproofs}.

Write $\mathcal{R}_{\rho}(G)$ for the full subcategory of $\mathcal{R}(G)$ consisting of smooth representations $M$ with the property that $M$ is generated by its $\rho$-isotypic component (the sum of all $K$-subrepresentations of $M$ isomorphic to $\rho$). The next theorem is not so hard, but gives some formal motivation for the more interesting Theorem \ref{fell3} and will be quite useful to us.

\begin{theorem}

If $G$ is a tame $\ell_c$-group, $K \subset G$ is any compact open subgroup, and $\rho \in \widehat{K}$, the functor of $\rho$-invariants that sends $M \mapsto \Hom_K(\rho,M|_K)$ restricts to an equivalence of categories \[ \mathcal{R}_{\rho}(G) \to \mathcal{H}(G,\rho)-\mathbf{mod}. \]

\label{type}

\end{theorem}

It is worth noting that if we do not assume that $G$ is an $\ell_c$-group then this theorem is known to be false in general. In the representation theory of $p$-adic reductive groups a pair $(K,\rho)$ for which the theorem holds is called a \emph{type}. 

\subsection{The Fell topology revisited}

Let $A$ be any unital associative ring, and recall that a \emph{left primitive ideal} of $A$ is the annihilator of a simple left $A$-module (one may have this discussion for non-unital rings, but we are working locally with unital rings anyway). We will write $\widehat{A}$ for the space of left primitive ideals of $A$. The \emph{Jacobson} or \emph{hull-kernel topology} on $\widehat{A}$ is defined as follows: for an arbitrary subset $S \subset \widehat{A}$ we define the \emph{kernel of} $S$ \[ J_S = \bigcap_{I \in S} I \] to be the intersection of the ideals in $S$. The closure of $S$ is the set of primitive ideals which contains $J_S$, and this closure operator defines the Jacobson topology on $\widehat{A}$.

Now take $A = \mathcal{H}(G,\rho)$. Clearly Theorem \ref{type} implies that the functor of $\rho$-invariants induces a bijection $\mathscr{V}(K,\rho) \to \widehat{A}$, and in fact one can show this directly even when the theorem fails, e.g. when $G$ is reductive and $(K,\rho)$ is not a type. Finally, we state our third characterization of the Fell topology.

\begin{theorem}

Let $G$ be a tame $\ell_c$-group. Then the bijection $\mathscr{V}(K,\rho) \to \widehat{A}$ which takes $\rho$-invariants is a homeomorphism, where $\widehat{A}$ is given the Jacobson topology.

\label{fell3}

\end{theorem}

We point out that in \cite{GK}, the theorem is proved for trivial $\rho$.

\subsection{The Bernstein center of $\mathcal{R}(G)$}

For any abelian category $\mathcal{A}$ we denote by $\mathscr{Z}(\mathcal{A})$ the Bernstein center of $\mathcal{A}$, that is, the commutative $\C$-algebra of endomorphisms of the identity functor on $\mathcal{A}$. For an $\ell$-group $G$ we may consider the Bernstein center of $\mathcal{R}(G)$, which we abusively denote as $\mathscr{Z}(G)$ and call the Bernstein center of $G$, or more generally, if $W \subset \widehat{G}$ we write $\mathscr{Z}_W(G)$ for $\mathscr{Z}(\mathcal{R}(G)_W)$. It is possible to say a few things about $\mathscr{Z}_W(G)$ in our situation, meaning when $G$ is a tame $\ell_c$-group. Notably, it may be considered as an algebra of $\C$-valued functions on $W$, since in light of Proposition \ref{schur} an element of $\mathscr{Z}_W(G)$ acts on each $\pi \in W$ by a scalar. In fact, if we write $\Fun(W)$ for the $\C$-algebra of $\C$-valued functions on $W$ under pointwise multiplication, then the map $\mathscr{Z}_W(G) \longrightarrow \Fun(W)$ thus defined is an algebra homomorphism. The following fact is an easy consequence of Lemma \ref{exact}.

\begin{proposition}

The natural algebra homomorphism $\mathscr{Z}_W(G) \longrightarrow \Fun(W)$ is injective.

\end{proposition}

In case $G$ is abelian and $W = \widehat{G}$, the image of this homomorphism is $C^{\infty}(\widehat{G})$, the algebra of locally constant $\C$-valued functions on $\widehat{G}$. This follows formally from the equivalence of $\mathcal{R}(G)$ with $\Sh(\widehat{G})$. Indeed, if $X$ is any topological space, it is well known that the Bernstein center of $\Sh(X)$ is $C^{\infty}(X)$. Our final theorem is a generalization of this.

\begin{theorem}

If $G$ is a tame $\ell_c$-group and $W \subset \widehat{G}$ is locally closed, then the image of the natural homomorphism $\mathscr{Z}_W(G) \longrightarrow \Fun(W)$ is $C^{\infty}(W)$.

\label{bernstein}

\end{theorem}

The rest of this paper consists of some additional technical exposition and the proofs of the theorems and lemmas stated above.

\section{More on Hecke algebras}

Before proving our first characterization of the Fell topology, we include a general discussion of Hecke algebras. For now $G$ is any $\ell$-group, and we must also fix a left Haar measure on $G$.

Recall that we have written $\mathcal{H}(G)$ for the Hecke algebra of $G$. If $K \subset G$ is a compact open subgroup, the subalgebra of bi-$K$-invariant functions will be denoted by $\mathcal{H}(G,K)$. Explicitly, $\mathcal{H}(G,K)$ consists of those $f \in \mathcal{H}(G)$ such that for any $g \in G$ and $k_1,k_2 \in K$, we have $f(k_1gk_2) = f(g)$. It is easy to show that $\mathcal{H}(G) = \bigcup \mathcal{H}(G,K)$, where the union runs over all compact open subgroups $K \subset G$: every function in the Hecke algebra is just a finite linear combination of indicator functions of double cosets $KgK$ for some sufficiently small compact open subgroup $K \subset G$.

The subalgebra $\mathcal{H}(G,K)$ has a unity element, namely $e_K = \frac{1}{\mu(K)} \cdot \mathbbm{1}_K$, where $\mathbbm{1}_K$ is the indicator function of $K$. In fact, $\mathcal{H}(G,K) = e_K * \mathcal{H}(G) * e_K$. Thus if $(\pi,V)$ is any smooth representation of $G$, the subspace of $K$-invariant vectors $V^K = e_k \cdot V$ is naturally an $\mathcal{H}(G,K)$-module. In particular, if $\pi$ is irreducible then either $V^K = 0$ or $V^K$ is a simple $\mathcal{H}(G,K)$-module (this fact is not specific to our situation, and holds much more generally for modules over an idempotented algebra).

Now if $\rho \in \widehat{K}$ we can define its character $\chi_{\rho} : K \to \C$ in the usual way, as irreducible representations of compact $\ell$-groups are finite-dimensional. Since $\rho$ factors through a representation of a finite quotient of $K$, we can use the orthogonality relations for characters of finite groups to see that $e_{\rho} = \frac{\chi_{\rho}(1)}{\mu(K)} \cdot \chi_{\rho}$ is an idempotent in the Hecke algebra. Moreover, just as in the finite case, if $(\pi,V)$ is any representation of $G$ then $e_{\rho} \cdot V$ is the sum of all $K$-subrepresentations of $\pi$ which are isomorphic to $\rho$. Observe also that $e_{\rho} \cdot V$ is naturally a module for the unital subalgebra $e_{\rho} * \mathcal{H}(G) * e_{\rho}$ of $\mathcal{H}(G)$, and that if $\pi$ is irreducible then $e_{\rho} \cdot V = 0$ or $e_{\rho} \cdot V$ is a simple module for $e_{\rho} * \mathcal{H}(G) * e_{\rho}$. If $\rho$ is trivial then $e_{\rho} = e_K$, so these remarks generalize the previous paragraph.

Recall that in Section 1 we defined a certain Hecke algebra $\mathcal{H}(G,\rho)$ which plays a role in our third characterization of the Fell topology. In fact, there is a canonical isomorphism of $\C$-algebras \[ \mathcal{H}(G,\rho) \otimes_{\C} \End_{\C}(W) \longrightarrow e_{\rho} * \mathcal{H}(G) * e_{\rho}, \] which means the algebras $\mathcal{H}(G,\rho)$ and $e_{\rho} * \mathcal{H}(G) * e_{\rho}$ are Morita equivalent. Note that these two algebras may be identified when $\rho$ is trivial (then they both coincide with $\mathcal{H}(G,K)$), but otherwise $\mathcal{H}(G,\rho)$ is not a subalgebra of $\mathcal{H}(G)$. 

\section{The proof of Theorem \ref{Fell}}

\label{topology}

Now we tie up the first loose thread from Section \ref{intro} and prove the characterization of the Fell topology stated in Theorem \ref{Fell}. Throughout this section $G$ will be any tame $\ell_c$-group.

There is a canonical choice of topology on the dual space $\widehat{G}$: if $(\pi,V) \in \widehat{G}$ pick a finite collection of vectors $v_1,\cdots,v_n \in V$, another $\xi_1,\cdots,\xi_n \in \widetilde{V}$, a compact subset $B \subset G$, and $\epsilon > 0$, then let $\mathscr{U}(\pi,v_j,\xi_j,B,\epsilon)$ consist of those $(\rho,W) \in \widehat{G}$ for which we can find $w_1,\cdots,w_n \in W$ and $\zeta_1,\cdots,\zeta_n \in \widetilde{W}$ satisfying \[ \big| \ \langle \xi_j, \pi(g)(v_j) \rangle - \langle \zeta_j, \rho(g)(w_j) \rangle \ \big| < \epsilon \] for all $1 \leq j \leq n$ and $g \in B$. The collection of all such $\mathscr{U}$'s forms a basis of neighborhoods for the Fell topology on $\widehat{G}$, which we have characterized in a way which we hope the reader finds more understandable.

The following technical lemma is critical in our proof of Theorem \ref{Fell}.

\begin{lemma}

Fix $(\pi,V) \in \widehat{G}$ and a nonzero finite-dimensional subspace $W \subset V$. Then there exists a compact open subgroup $K \subset G$ such that the $K$-subrepresentation generated by $W$ is irreducible.

\label{lemma1}

\end{lemma}

\begin{proof}

Choose a compact open subgroup $L \subset G$ such that $W \subset V^L$, so we know that $V^L$ is a simple $\mathcal{H}(G,L)$-module. Moreover, $V^L$ is finite-dimensional since $G$ is tame. Observe that $\mathcal{H}(G,L)$ is a filtered union of finite-dimensional subalgebras $\mathcal{H}(K,L)$, where $K$ ranges over compact open subgroups of $G$ containing $L$, by the $\ell_c$ assumption. Now let \[ \varphi : \mathcal{H}(G,L) \longrightarrow \End_{\C}(V^L) \] be the action map and put $A = \im \varphi$ (in fact, $\varphi$ must be surjective, but this is not necessary for the proof). Then $A$ is a finite-dimensional quotient of $\mathcal{H}(G,L)$, so there exists a compact open subgroup $K \supset L$ such that $\mathcal{H}(K,L)$ surjects onto $A$. But this means $V^L$ is already a simple $\mathcal{H}(K,L)$-module.

Now let $(\rho,X)$ be the $K$-subrepresentation generated by $V^L$ (so if we can show that $\rho$ is irreducible, $W$ generates $X$ also). Suppose $X$ decomposes as a $K$-representation into $X_1 \oplus X_2$. Then \[ V^L = X^L = X_1^L \oplus X_2^L \] implies $X_1^L = 0$ or $X_2^L = 0$ by simplicity of $V^L$, whence $V^L \subset X_1$ or $V^L \subset X_2$. Thus $X_1 = X$ or $X_2 = X$ as desired, since $V^L$ generates $X$.

\end{proof}

It seems natural to mention here one consequence of the lemma, which does not play a role in the proof of Theorem \ref{Fell} but will help us later in this paper.

\begin{lemma}

For any $\pi \in \widehat{G}$, we can find a compact open subgroup $K \subset G$ and $\rho \in \widehat{K}$ such that $| \rho : \pi|_K | = 1$, which is to say the space $\Hom_K(\rho,\pi|_K)$ is one-dimensional.

\end{lemma}

\begin{proof}

First find some compact open subgroup $L \subset G$ such that the space of $L$-invariants $\pi^L$ is nonzero. Then apply Lemma \ref{lemma1} to find a compact open subgroup $K$ such that the $K$-subrepresentation $\rho$ of $\pi$ generated by $\pi^L$ is irreducible. In particular, $\rho^L \neq 0$, and if $\rho \to \pi|_K$ is a $K$-morphism it must map $\rho^L$ into $\pi^L \subset \rho$, whence the space of such morphisms is one-dimensional by Schur's lemma.

\end{proof}

Now we are in a position to prove the theorem.

\begin{proof}[Proof of Theorem \ref{Fell}]

First we show that $\mathscr{V}(K,\rho)$ is open. Resuming our notation in the statement of the theorem, choose $(\pi,V) \in \mathscr{V}(K,\rho)$. By our remarks in section 1.2 there exists $v \in V$ such that $e_{\rho} \cdot v \neq 0$, and certainly there also exists $f \in \widetilde{V}$ such that $f(e_{\rho} \cdot v) \neq 0$. Now put $L = \ker \rho$ and $C = \max_{k \in K/L} |\chi_{\rho}(k)|$, and let $0 < \epsilon < \frac{|f(e_{\rho} \cdot v)|}{C \cdot \chi_{\rho}(1)}$. Then it is not hard to see by a direct calculation that $\mathscr{U}(\pi,v,f,K,\epsilon) \subset \mathscr{V}(K,\rho)$.

Now to prove that the $\mathscr{V}$'s form a basis for the Fell topology, let $(\pi,V) \in \widehat{G}$ be arbitary and choose $v_1,\cdots,v_n \in V$, $f_1,\cdots,f_n \in \widetilde{V}$, compact $B \subset G$, and $\epsilon > 0$. This defines a Fell neighborhood $\mathscr{U}(\pi,v_j,f_j,B,\epsilon)$ of $\pi$. Let $W$ be the $\C$-span of the $v_1,\cdots,v_n$ and apply Lemma \ref{lemma1} to find a compact open subgroup $K \subset G$ such that the $K$-subrepresentation $\rho$ generated by $W$ is irreducible. Note that we can use the fact that $G$ is an $\ell_c$-group to choose $K$ large enough so that $B \subset K$. An easy computation confirms that $\mathscr{V}(K,\rho) \subset \mathscr{U}(\pi,v_j,f_j,B,\epsilon)$.

\end{proof}

\section{Proving the results on stalks in the category $\mathcal{R}(G)$}

\label{stalkfunct}

In this section we prove the assertions made about the functors $\pi \mapsto M(\pi)$ for $\pi \in \widehat{G}$, with the goal of applying some sheaf-theoretic techniques to smooth representations. These results are independent of our previous work but play an important role in the proofs that follow. Recall that we denote the largest quotient of a smooth representation $M$ supported at $\pi$ by $M_{\pi}$, and that $M(\pi)$ is the ``multiplicity space" of $\pi$ in $M_{\pi}$.

First we need some technical consequences of the $\ell_c$-condition. Recall that we can define the so-called Jacquet functor $J_G : \mathcal{R}(G) \to \C-\mathbf{vect}$ by sending $M \in \mathcal{R}(G)$ to the the space of coinvariants, i.e. the quotient of $M$ by the $\C$-span of $\{ g \cdot v - v \ | \ v \in M \}$. This extends to a right exact functor which is left adjoint to the functor $\C-\mathbf{\vect} \to \mathcal{R}(G)$ that sends a vector space $V$ to the trivial representation of $G$ on $V$.

\begin{proposition}

For an $\ell_c$-group $G$, the functor $J_G$ is exact.

\end{proposition}

Jacquet proved this proposition: see, for instance, \cite{BZ}.

\begin{lemma}

If $G$ is an $\ell_c$-group and $M \in \mathcal{R}(G)$ the smooth dual $M^{\vee}$ is an injective object in $\mathcal{R}(G)$.

\label{injective}

\end{lemma}

\begin{proof}

For any $N \in \mathcal{R}(G)$ we have natural isomorphisms 
\begin{align*}
\Hom_G(N,M^{\vee}) &= \Hom_G(N,(M^*)^{sm}) \cong \Hom_G(N,\Hom_{\C}(M,\C)) \\
&\cong \Hom_G(N \otimes_{\C} M,\C) \cong \Hom_{\C}(J_G(N \otimes_{\C} M),\C). 
\end{align*}
Then we see that $M \mapsto M^{\vee}$ is the composition of three exact functors, since vector spaces are flat and injective and $J_G$ is exact.

\end{proof}

Note that Proposition \ref{admissible} and Lemma \ref{injective} imply the following key fact.

\begin{proposition}

Irreducible smooth representations of tame $\ell_c$-groups are injective objects in $\mathcal{R}(G)$.

\label{inj}

\end{proposition}

At this point we can apply Zorn's lemma to see that our definition of support was a reasonable choice.

\begin{lemma}

If $M \in \mathcal{R}(G)$, then $M = 0$ if and only if $\Supp M = \varnothing$.

\label{emptysupp}

\end{lemma}

\begin{proof}

The ``only if" direction is trivial. For the other implication, it suffices to show that any nonzero $M \in \mathcal{R}(G)$ has an irreducible quotient. For this, choose a nonzero finitely generated subrepresentation $N \subset M$, whence an easy application of Zorn's lemma shows that $N$ has an irreducible quotient $\pi \in \widehat{G}$. Then use Proposition \ref{inj} to see that $\pi$ is in fact a quotient of $M$.

\end{proof}

\begin{corollary}

Let $f : M \to N$ be a morphism in $\mathcal{R}(G)$. Then $f = 0$ if and only if for each $\pi \in \widehat{G}$ and $g \in \Hom_G(N,\pi)$, we have $g \circ f = 0$.

\label{zeromap}

\end{corollary}

\begin{proof}

Again we need only prove the ``if" direction, and by the lemma it is enough to show that $\Supp \im f = \varnothing$. But this is clear, since every morphism in $\Hom_G(\im f,\pi)$ lifts to $N$ by injectivity of $\pi$.

\end{proof}

Here we prove Theorem \ref{skyscraper}, which roughly says that smooth representations are ``pointwise semisimple."

\begin{proof}[Proof of Theorem \ref{skyscraper}]

The functor $V \mapsto V_G \otimes \pi$ is obviously fully faithful, and to see that it is essentially surjective it is enough to show that every $M \in \mathcal{R}(G)_{\pi}$ is a direct sum of copies of $\pi$. For this, observe that the proof of Lemma \ref{emptysupp} can be adapted to show that the natural map \[ M \longrightarrow \prod_{\Hom_G(M,\pi)} \pi \] is a monomorphism: apply Zorn's lemma to prove that the subrepresentation generated by any nonzero $m \in M$ has an irreducible quotient, and use Proposition \ref{inj} to see that this subquotient of $M$ is in fact a quotient, necessarily isomorphic to $\pi$. More concisely, every nonzero element of $M$ survives in an irreducible quotient.

Now to finish the proof, we claim that if $I$ is any set, the smooth part of the product $\prod_I \pi$ is a sum of copies of $\pi$. But in fact the (clearly injective) map given below is an isomorphism.
\begin{align*}
\bigg( \prod_I \C \bigg) \otimes \pi & \longrightarrow \bigg( \prod_I \pi \bigg)^{sm} \\ 
\bigg( \prod_{i \in I} a_i,v \bigg) & \mapsto \prod_{i \in I} a_iv
\end{align*}
To see why, choose $v_i \in \pi$ for each $i \in I$ so that $\prod v_i$ is smooth, i.e. lies in the smooth part of $\prod_I \pi$. Then $G$ has a compact open subgroup $K$ such that $\prod v_i$ is contained in the subspace $(\prod_I \pi)^K = \prod_I \pi^K$ of $K$-invariants. Now $\dim_{\C} \pi^K < \infty$ by the tameness assumption, so we can choose a finite spanning set $x_1,\cdots,x_m$ for $\pi^K$. Finally, write $v_i = \sum_j a_{ij}x_j$ and observe that $\sum_j (\prod_i a_{ij}) \otimes x_j$ is sent to $\prod v_i$.

\end{proof}

Next we prove Lemma \ref{point}.

\begin{proof}[Proof of Lemma \ref{point}]

Notice that $M_{\pi}$ is defined precisely so that the natural morphism \[ M_{\pi} \to \prod_{f \in \Hom_G(M_{\pi},\pi)} \pi \] is a monomorphism. But the smooth part of $\prod \pi$ is a direct sum of copies of $\pi$ by the proof of Theorem \ref{skyscraper}.

\end{proof}

Recall that exactness of sequences in $\mathcal{R}(G)$ can be tested on stalks: let us prove Lemma \ref{exact}.

\begin{proof}[Proof of Lemma \ref{exact}]

To see that the functor $M \mapsto M(\pi)$ is exact, observe that there is a sequence of natural isomorphisms \[ \Hom_G(M,\pi) \cong \Hom_G(M_{\pi},\pi) \cong \Hom_G(M(\pi) \otimes \pi,\pi) \cong M(\pi)^*, \] where the first isomorphism uses the above-mentioned adjunction and the latter two use Proposition \ref{schur}. But the functor $M \mapsto \Hom_G(M,\pi)$ is exact by injectivity of $\pi$, which means that the the dual sequence \[ P(\pi)^* \longrightarrow N(\pi)^* \longrightarrow M(\pi)^* \] is exact, and consequently so is the sequence \[ M(\pi) \longrightarrow N(\pi) \longrightarrow P(\pi). \]

For the converse, we first prove that $g \circ f = 0$. Fix $\pi \in \widehat{G}$ and $h \in \Hom_G(P,\pi)$, so by Corollary \ref{zeromap} it suffices to show that $h \circ g \circ f = 0$. Well, since $(g \circ f)_{(\pi)} = g_{(\pi)} \circ f_{(\pi)} = 0$ it is easy to see that the induced map $M_{\pi} \to P_{\pi}$ is zero, and then we have what we need because $h$ factors through the natural projection $P \to P_{\pi}$. A similar sort of argument, which we leave to the reader, shows that $\Supp \ker g / \im f = \varnothing$ and consequently Lemma \ref{emptysupp} gives us exactness.

\end{proof}

\section{The proof of Theorem \ref{type} and a useful lemma}

\label{heckeproofs}

Let us review induction functors and their adjunction relations. Fix an $\ell$-group $G$ and a closed subgroup $H \subset G$. Then the functor $\Res_H^G : \mathcal{R}(G) \to \mathcal{R}(H)$ which restricts representations to $H$ has a right adjoint described as follows. Choose $(\pi,V) \in \mathcal{R}(H)$, and let \[ I(V) = \{ f : G \to V \ | \ f(hg) = \pi(h)(f(g)) \ \text{for all} \ h \in H,g \in G \} \] with a $G$-action defined by $(g \cdot f)(x) = f(xg)$. Then the induced representation $\Ind_H^G \pi$ is defined to be the smooth part of $I(V)$, and after extending to morphisms in the obvious way we call the resulting functor $\Ind_H^G : \mathcal{R}(H) \to \mathcal{R}(G)$ \emph{smooth induction}. Also, we define $\cInd_H^G \pi$ to be the subrepresentation of $\Ind_H^G \pi$ consisting of those functions with compact support modulo $H$, and call $\cInd_H^G : \mathcal{R}(H) \to \mathcal{R}(G)$ \emph{smooth induction with compact supports}, a subfunctor of $\Ind_H^G$. In case $H$ is open, $\cInd_H^G$ is left adjoint to $\Res_H^G$. 

\begin{proof}[Proof of Theorem \ref{type}]

By Proposition 3.3 in \cite{BK}, it suffices to prove that $\mathcal{R}_{\rho}(G)$ is closed under taking subquotients. But in fact this full subcategory of $\mathcal{R}(G)$ coincides with $\mathcal{R}(G)_{\mathscr{V}(G,\rho)}$, which is clearly closed under taking subquotients.

To see why, suppose $M \in \mathcal{R}_{\rho}(G)$ and consider the inclusion $e_{\rho} \cdot M \to M|_{K}$ as a $K$-morphism. Then there is a corresponding $G$-morphism $\cInd_K^G e_{\rho} \cdot M \to M$, which is surjective because $M$ is generated by $e_{\rho} \cdot M$, so that $\Supp M \subset \Supp (\cInd_K^G e_{\rho} \cdot M)$. But clearly $\cInd_K^G e_{\rho} \cdot M$ is supported on $\mathscr{V}(K,\rho)$, because for any $\pi \in \widehat{G}$ we have \[ \Hom_G(\cInd_K^G e_{\rho} \cdot M,\pi) \cong \Hom_K(e_{\rho} \cdot M,\pi|_K) \] and if $\pi \notin \mathscr{V}(K,\rho)$ then the latter space is zero.

Conversely, let $M \in \mathcal{R}(G)$ be supported on $\mathscr{V}(K,\rho)$, and write $N = \mathcal{H}(G) * e_{\rho} \cdot M$ for the $G$-subrepresentation of $M$ generated by the $\rho$-isotypic component $e_{\rho} \cdot M$. Obviously $\Supp M/N \subset \Supp M \subset \mathscr{V}(K,\rho)$, and also $e_{\rho} \cdot M \subset N$ implies $\Supp(M/N) \cap \mathscr{V}(K,\rho) = \varnothing$. Thus $\Supp M/N = \varnothing$, so by Lemma \ref{emptysupp} we have $M = N$ as desired.

\end{proof}

Here we include a technical lemma, which will in fact be used in all the proofs that follow. But first it is useful to introduce an auxiliary category, which we denote by $\check{G}$, in order to clean up the statement of the lemma. An object of $\check{G}$ is a pair $(K,\rho)$ consisting of a compact open subgroup $K \subset G$ and an irreducible smooth representation $\rho$ of $K$, and there are no morphisms $(L,\sigma) \to (K,\rho)$ unless $K \subset L$. If $K \subset L$, then a morphism $(L,\sigma) \to (K,\rho)$ in $\check{G}$ is a nonzero (hence surjective) $K$-morphism $\sigma|_K \to \rho$. As the notation suggests, the category $\check{G}$ is closely related to the topological space $\widehat{G}$.

Now each $M \in \mathcal{R}(G)$ gives rise to a presheaf $\mathscr{M}$ on $\check{G}$ with values in complex vector spaces. Define \[ \mathscr{M}(K,\rho) = \Hom_K(\rho,M|_K), \] and we obtain the ``restriction" map associated with a morphism $(L,\sigma) \to (K,\rho)$ as follows. There is a canonical isomorphism \[ \mathscr{M}(K,\rho) = \Hom_K(\rho,M|_K)  \to \Hom_L(\cInd_K^L \rho,M|_L), \] and the chosen $K$-morphism $\sigma|_K \to \rho$ corresponds to an $L$-morphism \[ \sigma \to \Ind_K^L \rho = \cInd_K^L \rho, \] which induces a pullback map \[ \Hom_L(\cInd_K^L \rho,M|_L) \to \Hom_L(\sigma,M|_L) = \mathscr{M}(L,\sigma). \] It is not so hard to see that these ``restrictions" are surjective. Now we state and prove the aforementioned lemma, the proof of which uses some ideas from \cite{GK}.

\begin{lemma}

Let $M \in \mathcal{R}(G)$ and suppose $\pi \in \widehat{G}$ but $\pi \notin \Supp M$. Then if $(K,\rho) \in \widehat{G}$ with $\pi \in \mathscr{V}(K,\rho)$, for any $\alpha \in \mathscr{M}(K,\rho)$ there exists $(L,\sigma) \in \check{G}$ and a morphism $(L,\sigma) \to (K,\rho)$ such that $\pi \in \mathscr{V}(L,\sigma)$ and $\alpha$ is annihilated by the induced map $\mathscr{M}(K,\rho) \to \mathscr{M}(L,\sigma)$. If $(K,\rho)$ is chosen so that $| \rho : \pi|_K | = 1$, then $(L,\sigma)$ can be chosen so that $| \sigma : \pi|_L | = 1$ and the morphism $(L,\sigma) \to (K,\rho)$ is unique modulo nonzero scalars.

\label{shrink}

\end{lemma}

\begin{proof}

Let us write $A = \mathcal{H}(G,\rho)$, so $\widetilde{M} = \mathscr{M}(K,\rho)$ is naturally an $A$-module, as is $\widetilde{\pi} = \Hom_K(\rho,\pi|_K)$, and the latter is finite-dimensional and simple (in particular, nonzero). In view of Theorem \ref{type} it is clear that $\Hom_A(\widetilde{M},\widetilde{\pi}) = 0$, and in particular $\Hom_A(A \cdot \alpha,\widetilde{\pi}) = 0$ since $\widetilde{\pi}$ is injective in $A - \mathbf{mod}$. If we put $I = \Ann_A(\widetilde{\pi})$ and $J = \Ann_A(\alpha)$, this means that $I + J = A$, and in particular we can find $x \in I$ and $y \in J$ such that $x + y = 1$. Now by the $\ell_c$ condition $A$ is a filtered union of subalgebras $\mathcal{H}(L,\rho)$, where $L$ is a compact open subgroup containing $K$, whence it is possible to find $L \supset K$ large enough that $x,y \in \mathcal{H}(L,\rho)$ and $\widetilde{\pi}$ is a simple $\mathcal{H}(L,\rho)$-module. Since the space \[ \Hom_L(\cInd_K^L \rho,\pi|_L) = \Hom_K(\rho,\pi|_K) \] is nonzero by assumption, there is some $\sigma \in \widehat{L}$ such that $| \sigma : \cInd_K^L \rho |, | \sigma : \pi|_L | > 0$. This means we can also find some nonzero $L$-morphism $\sigma \to \cInd_K^L \rho$, which corresponds to the needed $K$-morphism $\sigma|_K \to \rho$. If we chose $\rho$ so that $| \rho : \pi|_K | = 1$, then now we have a unique choice of $\sigma \in \widehat{L}$ with $| \sigma : \cInd_K^L \rho |, | \sigma : \pi|_L | = 1$, and the morphism $\sigma|_K \to \rho$ is also unique modulo nonzero scalars.

Now we show that the morphism $(L,\sigma) \to (K,\rho)$ we have found has the property that $\alpha$ is annihilated by the ``restriction" $\mathscr{M}(K,\rho) \to \mathscr{M}(L,\sigma)$. Let us write $B = \mathcal{H}(L,\rho)$ and observe that $I \cap B + J \cap B = B$ (because $x+y=1$), which says that $\Hom_B(B \cdot \alpha, \widetilde{\pi}) = 0$. To say that $\alpha$ vanishes in the restriction is to say that the composite $L$-morphism \[ \sigma \to \cInd_K^L \rho \to M|_L \] is zero, so let us assume otherwise and find a contradiction. In that case we can certainly find an $L$-morphism $\varphi : M|_L \to \pi|_L$ such that the composition \[ \sigma \to \cInd_K^L \rho \to M|_L \to \pi|_L \] is nonzero. But then this induces a $B$-morphism $\varphi_* : \widetilde{M} \to \widetilde{\pi}$ with $\varphi_*(\alpha) \neq 0$, and we just saw that $\Hom_B(B \cdot \alpha, \widetilde{\pi}) = 0$.

\end{proof}

\begin{corollary}

In the situation of the lemma, if we assume also that $M$ is admissible, it is possible to find $(L,\sigma) \to (K,\rho)$ so that $\mathscr{M}(L,\sigma) = 0$.

\end{corollary}

\begin{proof}

If $M$ is admissible then $\dim_{\C} \mathscr{M}(K,\rho) < \infty$, so the corollary follows from the lemma in view of the fact that the ``restriction" maps $\mathscr{M}(K,\rho) \to \mathscr{M}(L,\sigma)$ are surjective.

\end{proof}

\section{The proof of Theorem \ref{fell2}}

\label{adjointproofs}

Next we will prove Theorem \ref{fell2}, but first let us recall some notation. Let $Z \subset W \subset \widehat{G}$ be any subsets and $U = W \setminus Z$, and write $i_! : \mathcal{R}(G)_Z \to \mathcal{R}(G)_W$ and $j_! : \mathcal{R}(G)_U \to \mathcal{R}(G)_W$ for the inclusions. Now $j_!$ always has a right adjoint $j^!$ for formal reasons, given by the formula \[ j^!M = \bigcap_{\substack{f \in \Hom_G(M,\pi) \\ \pi \in Z}} \ker f \] on objects $M \in \mathcal{R}(G)_W$. It is easy to see that this construction is functorial. We can define $i^*M = M/j^!M$ or more formally $M/j_!j^!M$, and it is clear enough that if $i_!$ has a left adjoint it must be given by this formula. The subtle issue is whether $i^*M$ is always supported on $Z$, and the theorem says that as long as $W$ is locally closed, this happens precisely when $Z$ is closed in $W$.

\begin{proof}[Proof of Theorem \ref{fell2}]

First suppose that $Z$ is closed in $W$ and $M \in \mathcal{R}(G)_W$. To see that $\Supp i^*M \subset Z$, fix $\pi \in \Supp i^*M$ and $(K,\rho) \in \check{G}$ such that $\pi \in \mathscr{V}(K,\rho)$, so it suffices to prove that $Z \cap \mathscr{V}(K,\rho) \neq \varnothing$ because then $\pi \in \overline{Z} = Z$. These assumptions give us nonzero (hence surjective) morphisms $i^*M \to \pi$ and $\pi|_K \to \rho$, which we can compose to get a nonzero morphism $i^*M|_K \to \pi|_K \to \rho$. Then by semisimplicity of $\mathcal{R}(K)$ we can find a nonzero morphism $\rho \to M|_K$ which does not land in $j^!M$, so by our definition of the latter there exists $\zeta \in Z$ and a morphism $M \to \zeta$ such that the composition $\rho \to M|_K \to \zeta|_K$ is nonzero. That is, $\zeta \in Z \cap \mathscr{V}(K,\rho)$.

The rest is more formal. For all $M \in \mathcal{R}(G)_W$ and $N \in \mathcal{R}(G)_Z$ we have a pullback map \[ \Hom_G(i^*M,N) \to \Hom_G(M,N), \] which we need to show is an isomorphism (the naturality requirement is clear). This map is injective because $M \to i^*M$ is surjective, so for the surjectivity fix a morphism $M \to N$. By Corollary \ref{zeromap}, to see that this morphism descends to the quotient it is enough to check that for any $\pi \in W$, the composition \[ j^!M \to M \to N \to \pi \] is zero. If $\pi \notin Z$, then this is just because $\Supp N \subset Z$, and if $\pi \in Z$ then use the definition of $j^!M$.

For the other implication, we can immediately reduce to the case that $W$ is open in $\widehat{G}$: let us assume for the moment that we have proved the theorem in this situation. Find open $V \supset W$ such that $W$ is closed in $V$, so by the first implication the inclusion $\mathcal{R}(G)_W \to \mathcal{R}(G)_V$ has a left adjoint. But then if $\mathcal{R}(G)_Z \to \mathcal{R}(G)_W$ has a left adjoint we can compose the two to get a left adjoint for $\mathcal{R}(G)_Z \to \mathcal{R}(G)_V$, and the special case implies that $Z$ is closed in $V$, hence closed in $W$.

So assume $W$ is open. If $i_!$ has a left adjoint, it is easy to check that it must be isomorphic to $i^*$, and in particular $i^*M$ is supported on $Z$ for every $M \in \mathcal{R}(G)_W$. Take $M \subset \mathcal{H}(G)$ to be the ``restriction" of the left regular representation to $W$: explicitly, if $k : W \to \widehat{G}$ is the open inclusion then $M = k^!\mathcal{H}(G)$, where $\mathcal{H}(G)$ is given the $G$-action by translations. Clearly $\Supp M = W$, so it makes sense to consider $i^*M$, and with our assumption we get $\Supp i^*M = Z$. In particular, if $\pi \in W \setminus Z$ then $\Hom_G(M,\pi) = 0$.

Find $(K,\rho) \in \check{G}$ such that $\pi \in \mathscr{V}(K,\rho) \subset W$, and as before put $A = \mathcal{H}(K,\rho)$. If we restrict $M$ further to $\mathscr{V}(K,\rho)$, then, under the equivalence of Theorem \ref{type}, $M$ corresponds to $\mathscr{M}(K,\rho) = A$ viewed as an $A$-module. Thus we can apply Lemma \ref{shrink} to find $(L,\sigma) \in \check{G}$ and a morphism $(L,\sigma) \to (K,\rho)$ such that $\pi \in \mathscr{V}(L,\sigma) \subset \mathscr{V}(K,\rho)$ and $1 \in A$ is annihilated by the restriction map $\mathscr{M}(K,\rho) \to \mathscr{M}(L,\sigma)$. But then this $A$-map is zero, and we know it to be surjective also, whence $\mathscr{M}(L,\sigma) = 0$ or equivalently $\mathscr{V}(L,\sigma) \cap Z = \varnothing$ and $Z$ is closed as desired.

Now for the exactness, notice that when $i^*$ exists it is right exact simply because it is a left adjoint, so it suffices to show that $i^*$ preserves monomorphisms. Let $f : M \to N$ be an injective morphism in $\mathcal{R}(G)_W$, write $P \subset i^*M$ for the kernel of $i^*f$, and fix a morphism $g : P \to \pi$ for some $\pi \in W$. Then by injectivity we can lift $g$ to $i^*M \to \pi$, so if $\pi \notin Z$ then already this morphism is zero. Otherwise this naturally induces a morphism $M \to \pi$, and then use injectivity again to lift to $N \to \pi$. Since $\pi \in Z$ this descends to $\widetilde{g} : i^*N \to \pi$, whence $g(P) = \widetilde{g}(i^*f(P)) = 0$ and $P = 0$ by Lemma \ref{emptysupp}.

\end{proof}

We include a proof of Corollary \ref{projective} here.

\begin{proof}[Proof of Corollary \ref{projective}]

To see that an open point $\pi \in \widehat{G}$ is projective, write $i : \{ \pi \} \to \widehat{G}$ for the inclusion and observe that for any $M \in \mathcal{R}(G)$, the composition \[ i^!M \to M \to i^*M \] is an isomorphism because $\pi$ is open and closed. Notice $i^*M = M_{\pi}$ with the notation from section \ref{stalkstatements}, so \[ \Hom_G(\pi,M) = \Hom_G(\pi,i^!M) = \Hom_G(\pi,M_{\pi}) = M(\pi) \] and $M \to \Hom_G(\pi,M)$ is just the functor of taking stalks at $\pi$. But this functor is exact by Lemma \ref{exact}.

Now for the converse, in view of Theorem \ref{fell2} it suffices to prove that if $j : \widehat{G} \setminus \{ \pi \} \to \widehat{G}$, the inclusion $j_! : \mathcal{R}(G)_{\widehat{G} \setminus \{ \pi \}} \to \widehat{G}$ has a left adjoint. Moreover, by the proof of the theorem it is enough to check that $j^*M$ is supported on $\widehat{G} \setminus \{ \pi \}$ for all $M \in \mathcal{R}(G)$. Consider the short exact sequence \[ 0 \to \Hom_G(\pi,i^!M) \to \Hom_G(\pi,M) \to \Hom_G(\pi,j^*M) \to 0 \] (so here we are using the projectivity of $\pi$). But $\Hom_G(\pi,i^!M) \to \Hom_G(\pi,M)$ is obviously an isomorphism, which forces $\Hom_G(\pi,j^*M) = 0$. Fix a morphism $j^*M \to \pi$, which induces a map \[ \Hom_G(\pi,j^*M) \to \Hom_G(\pi,\pi) \cong \C, \] and this is obviously not surjective since the first space is zero. So $j^*M \to \pi$ is not surjective, hence zero, and $\Supp j^*M \subset \widehat{G} \setminus \{ \pi \}$ as desired.

\end{proof}

\section{The proof of Theorem \ref{fell3}}

Now we are ready to prove Theorem \ref{fell3}, which we will deduce from Theorem \ref{fell2}.

\begin{proof}[Proof of Theorem \ref{fell3}]

We start by proving that the bijection $\mathscr{V}(K,\rho) \to \widehat{A}$ is closed. Suppose $Z \subset \mathscr{V}(K,\rho)$ is closed, so to prove that its image in $\widehat{A}$ is closed we must show that if $I \in \widehat{A}$ contains \[ J_Z = \bigcap_{\zeta \in Z} \Ann_{A}(\widetilde{\zeta}) \] (here again we write $\widetilde{\zeta} = \Hom_K(\rho,\zeta|_K)$), then $\pi \in Z$ where $I = \Ann_{A}(\widetilde{\pi})$. If we take $W = \mathscr{V}(K,\rho)$ in the setup of Theorem \ref{fell2} and let $M \in \mathcal{R}(G)_W$ be such that $\widetilde{M} = A$ as an $A$-module, then $\widetilde{i^*M} = A/J_Z$ and in particular $A/J_Z$ is supported on the image of $Z$ in $\widehat{A}$, in the sense that $\Hom_A(A/J_Z,\widetilde{\zeta}) = 0$ for any $\zeta \in W \setminus Z$. The assumption $I \supset J_Z$ says that there is a canonical surjection $A/J_Z \to \widetilde{\pi}$, whence $\pi \in Z$.

Conversely, suppose $Z \subset \widehat{A}$ is closed in the Jacobson topology and let $J_Z$ be as before. We want to use Theorem \ref{fell2} again, so $W = \mathscr{V}(K,\rho)$ still, and it is enough to show that the inclusion $\mathcal{R}(G)_Z \to \mathcal{R}(G)_W$ has a left adjoint. Recall from the proof of Theorem \ref{fell2} that actually it suffices just to check that $i^*M$ is supported on $Z$ for all $M \in \mathcal{R}(G)_W$, so for this observe that $\widetilde{i^*M} = \widetilde{M}/J_Z\widetilde{M}$. But since $Z$ is closed in $\widehat{A}$, we have $\Hom_A(A/J_Z,\widetilde{\zeta}) = 0$ for $\zeta \in W \setminus Z$, which implies that $\widetilde{M}/J_Z\widetilde{M}$ and hence $i^*M$ is supported on $Z$.

\end{proof}

\section{The proof of Theorem \ref{bernstein}}

Finally, we come to Theorem \ref{bernstein}, which describes the Bernstein center of $\mathcal{R}(G)$. We will need the exactness of induction, which is proved, for example, in \cite{BZ}.

\begin{proposition}

For any $\ell$-group $G$ and closed subgroup $H$, the functors $\cInd_H^G,\Ind_H^G : \mathcal{R}(H) \to \mathcal{R}(G)$ are exact.

\end{proposition}

\begin{proof}[Proof of Theorem \ref{bernstein}]

First we show that the image of the homomorphism $\mathscr{Z}_W(G) \to\Fun(W)$ contains $C^{\infty}(W)$. Fix $f \in C^{\infty}(W)$ and decompose $W = \coprod_{i \in I} U_i$ into the level sets of $f$, which are open and closed in $W$. Given $M \in \mathcal{R}(G)_W$, let us write $M_{U_i}$ for the largest quotient of $M$ supported on $U_i$ obtained from Theorem \ref{fell2}, so writing $i_{U_i} : U_i \to W$ for the inclusions this means $M_{U_i} = i^*_{U_i}M$. In fact, since each $U_i$ is also open, the composition \[ i_{U_i}^!M \to M \to i_{U_i}^*M \] is an isomorphism and $M_{U_i}$ may be viewed as a subrepresentation in a canonical way. We claim that the natural morphism \[ \bigoplus_{i \in I} M_{U_i} \longrightarrow M \] is in fact an isomorphism, which we prove by using Lemma \ref{exact} to test the morphism on stalks. Fix $\pi \in W$, and recall from the proof of that lemma that $M(\pi)^* = \Hom_G(M,\pi)$, which is easier to work with. Therefore we need only observe that the induced map \[ \Hom_G(M,\pi) \longrightarrow \Hom_G \big( \bigoplus_{i \in I} M_{U_i},\pi \big) \cong \prod_{i \in I} \Hom_G(M_{U_i},\pi) \] is an isomorphism, so use Theorem \ref{fell2} and the fact that $\pi \in M_{U_i}$ for exactly one $i \in I$.

Now we define $\varphi_M \in \End_G(M)$ by letting $\varphi_M$ scale each summand $M_{U_i}$ by $f(U_i)$. To see that this is natural in $M$, choose a morphism $g : M \to N$ in $\mathcal{R}(G)$ and observe that by the proof of Theorem \ref{fell2}, $g$ sends each $M_{U_i}$ into $N_{U_i}$. Then the square we want to commute decomposes into a direct sum, and each summand commutes by $\C$-linearity of $g|_{M_{U_i}}$. Now since $\varphi_{\pi}$ is the scalar $f(\pi)$ it is easy to see that the homomorphism $\mathscr{Z}_W(G) \to \Fun(W)$ sends $\varphi$ to $f$ as needed.

As for the converse, choose $\theta \in \mathscr{Z}_W(G)$, which is a functorial family of endomorphisms $\theta_M : M \to M$ for all $M \in \mathcal{R}(G)_W$. Fix $\pi \in W$, so we need to find an open neighborhood of $\pi$ in $W$ such that $\theta$ acts on all points of this neighborhood by the same scalar. To this end we can assume without loss of generality that $\theta_{\pi} = 0$, and let us also choose $(K,\rho) \in \check{G}$ such that $| \rho : \pi|_K | = 1$ and $Z = \mathscr{V}(K,\rho) \cap W$ is closed in $\mathscr{V}(K,\rho)$. Let us write $i_! : \mathcal{R}(G)_Z \to \mathcal{R}_{\rho}(G)$ for the inclusion functor, which has an exact left adjoint $i^*$ by Theorem \ref{fell2}, and put $N = \cInd_K^G \rho \in \mathcal{R}_{\rho}(G)$. There is a unique (up to scaling) morphism $N \to \pi$, which factors through $i^*N$ because $\pi \in Z$, so denote by $M$ the kernel of the surjection $i^*N \to \pi$. Notice that $\theta_{\pi} = 0$ implies that $\theta_{i^*N}$ lands in $M$, and clearly $\Hom_G(M,\pi) = 0$. Thus we obtain a nonzero morphism $\alpha : \rho \to M|_K$ corresponding to the composition $N \to i^*N \to \pi$, so we can apply Lemma \ref{shrink} to find $(L,\sigma) \in \check{G}$ and a morphism $(L,\sigma) \to (K,\rho)$ such that $| \sigma : \pi|_K | = 1$, $\mathscr{V}(L,\sigma) \subset \mathscr{V}(K,\rho)$, and $\alpha$ is annihilated in the restriction $\mathscr{M}(K,\rho) \to \mathscr{M}(L,\sigma)$.

We claim that $\theta_{\zeta} = 0$ for any $\zeta \in \mathscr{V}(L,\sigma) \cap W$, so clearly it suffices to prove that for $P = \cInd_L^G \sigma$ we have $\theta_{i^*P} = 0$. To see the latter, consider the chosen injection $\sigma \to \cInd_K^L \rho$, which induces an injective $G$-morphism $\cInd_L^G \sigma \to \cInd_K^G \rho$ by exactness of $\cInd_L^G$, then look at the commutative square \[
\begin{diagram}
i^*P & \rTo^{\theta} & i^*P \\
\dTo &         & \dTo \\
i^*N  & \rTo^{\theta} & i^*N
\end{diagram}. \]
Then recall that $(L,\sigma) \to (K,\rho)$ was chosen to make the composition $i^*P \to i^*N \stackrel{\theta}{\to} i^*N$ zero, so from the injectivity of $i^*P \to i^*N$ the claim follows.

\end{proof}

\appendix

\section{The abelian case}

Here we discuss some details which relate to Theorem \ref{rodier}, and in particular give a proof which is different from Rodier's proof in \cite{Rod}. First we prove that all abelian $\ell_c$-groups are tame in a very strong sense.

\begin{lemma}

Any irreducible smooth representation of an abelian $\ell_c$-group is one-dimensional.

\end{lemma}

\begin{proof}

Let $G$ be an abelian $\ell_c$-group and pick $\pi \in \widehat{G}$. It is easy to see that $\pi$ factors through a discrete quotient of $G$, and a discrete group is $\ell_c$ if and only if it is torsion. But then for any $g \in G$, the operator $\pi(g)$ has finite order and hence is diagonalizable. Since $G$ is abelian, eigenspaces are $G$-stable, so by irreducibility of $\pi$ it follows that $G$ acts on the space of $\pi$ by scalars, whence the space must be one-dimensional.

\end{proof}

Of course this lemma gives us Schur's lemma trivially, without invoking Proposition \ref{schur}.

Next we state and prove the key result which, along with the theory of Fourier transforms on locally compact abelian groups, informs the proof of Theorem \ref{rodier}. If $X$ is any topological space, we write $C^{\infty}_c(X)$ for the algebra of locally constant complex-valued functions on $X$ with compact support under pointwise multiplication. Since this algebra is unital only when $X$ is compact, it is necessary to consider $C^{\infty}_c(X)$-modules which are nondegenerate as defined in subsection 1.6.

\begin{theorem}

Let $X$ be an $\ell$-space. The category of nondegenerate $C^{\infty}_c(X)$-modules is equivalent to the category $\Sh(X)$ of sheaves of complex vector spaces on $X$.

\label{modsheaf}

\end{theorem}

\begin{proof}

We will write $\mathfrak{m}_x$ for the maximal ideal of $C^{\infty}_c(X)$ consisting of functions which vanish at a fixed $x \in X$, and similarly $\mathfrak{m}_K$ denotes the ideal of functions which vanish on some compact open subset $K \subset X$. Given a nondegenerate $C^{\infty}_c(X)$-module $M$, it is natural to consider the vector space $M/\mathfrak{m}_KM$: this will be the space of sections over $K$. It is not hard to see that the associated presheaf whose sections locally have this form (recall that an $\ell$-space has a basis of compact open subsets) is a sheaf, which we call $\widetilde{M}$. Indeed, any compact open cover of a compact open set has a finite refinement consisting of pairwise disjoint compact open sets, so locally patching is easy. The stalk of $\widetilde{M}$ at a point $x \in X$ is just the space $M/\mathfrak{m}_xM$. It is not hard to check that the mapping $M \mapsto \widetilde{M}$ extends to a functor from the category of nondegenerate $C^{\infty}_c(X)$-modules to $\Sh(X)$, and we claim that this functor is an equivalence.

First, observe that $M \mapsto \widetilde{M}$ has a right adjoint which sends a sheaf $\mathcal{F} \in \Sh(X)$ to its space $\Gamma_c(X,\mathcal{F})$ of compactly supported global sections, with the obvious $C^{\infty}_c(X)$-action: the level sets of any $f \in C^{\infty}_c(X)$ are open. The unit morphisms $M \to \Gamma_c(X,\widetilde{M})$ are defined as follows: given $m \in M$, choose a compact open subset $K \subset X$ which contains $\Supp m = \{ x \in X \ | \ m \notin \mathfrak{m}_xM \}$, or equivalently $\mathbbm{1}_K \cdot m = m$. Then pass to the quotient $M/ \mathfrak{m}_K M = \widetilde{M}(K)$, from which we obtain $\widetilde{m} \in \Gamma_c(X,\widetilde{M})$ by extending by zero outside of $K$. This construction is easily seen to be independent of the choice of $K$. Now pick $\mathcal{F} \in \Sh(X)$: the counit is a morphism of sheaves, so we define it locally (over compact open sets). For any compact open subset $K \subset X$, consider the restriction map $\Gamma_c(X,\mathcal{F}) \to \mathcal{F}(K)$. This vanishes on $\mathfrak{m}_K \Gamma_c(X,\mathcal{F})$, and is clearly compatible with the restriction maps of $\mathcal{F}$ and $\widetilde{\Gamma_c(X,\mathcal{F})}$. We leave to the reader the verification that the unit-counit relations are satisfied.

But in fact the unit and counit are isomorphisms. It is clear enough that $M \to \Gamma_c(X,\widetilde{M})$ is surjective: if $s \in \Gamma_c(X,\widetilde{M})$ choose compact open $K \subset X$ containing $\Supp s$ and pick a representative $m \in M$ for $s|_K \in \widetilde{M}(K) = M/ \mathfrak{m}_K M$, then note that $\mathbbm{1}_K \cdot m$ is sent to $s$ by checking locally. Injectivity of the unit is much harder: observe that it suffices to prove that for $m \in M$, we have $\Supp m = \varnothing$ if and only if $m = 0$. The nontrivial part is showing that if $m \in \mathfrak{m}_xM$ for all $x \in X$, then $m = 0$. First pick a compact open subset $K \subset X$ such that $\mathbbm{1}_K \cdot m = m$, which we know is possible because $M$ is nondegenerate. Then we can identify $\mathbbm{1}_K \cdot C^{\infty}_c(X)$ with the (unital) algebra $C^{\infty}_c(K)$, and the maximal ideals of $C^{\infty}_c(K)$ all have the form $\mathbbm{1}_K \cdot \mathfrak{m}_x$ for $x \in K$. Moreover, $\mathbbm{1}_K \cdot M$ is a (unital) module for $C^{\infty}_c(K)$, and we just need to show $m = \mathbbm{1}_K \cdot m = 0$, so by all of this we are reduced to the case that $X$ is compact. Thus we are dealing with an ordinary unital, commutative $\C$-algebra $A = C^{\infty}_c(X)$, and it is not hard to see that for each $x \in X$ the localization map $A \to A_{\mathfrak{m}_x}$ is surjective with kernel $\mathfrak{m}_x$, so we have an isomorphism $A/\mathfrak{m}_xA \to A_{\mathfrak{m}_x}$ (actually both are canonically identified with $\C$). But it is a well-known fact of commutative algebra that if $m$ is zero in every localization $M_{\mathfrak{m}_x}$ then $m = 0$.

As for the counit, notice first that that the canonical maps $\Gamma_c(X,\mathcal{F}) \to \mathcal{F}_x$ for $x \in X$ are all surjective since $X$ has a basis of compact open subsets. This shows that the counit is an epimorphism, and it is just as clear that if $s \in \Gamma_c(X,\mathcal{F})$ goes to zero in $\mathcal{F}_x$ then $s \in \mathfrak{m}_x\Gamma_c(X,\mathcal{F})$. Thus the counit induces an isomorphism on stalks, so it is an isomorphism.

\end{proof}

With this out of the way, we are in a position to prove Theorem \ref{rodier}.

\begin{proof}[Proof of Theorem \ref{rodier}]

Let $G$ be an abelian $\ell_c$-group. First, recall that $\mathcal{R}(G)$ is isomorphic to the category of nondegenerate $\mathcal{H}(G)$-modules. The Fourier transform provides an isomorphism of algebras $\mathcal{H}(G) \to C^{\infty}_c(\widehat{G})$, so $\mathcal{R}(G)$ is in fact isomorphic to the category of nondegenerate $C^{\infty}_c(\widehat{G})$-modules. Now the $\ell_c$ assumption on $G$ implies that $\widehat{G}$ is totally disconnected and hence an $\ell$-space, so we can apply Theorem \ref{modsheaf} to see that the category of nondegenerate $C^{\infty}_c(\widehat{G})$-modules is equivalent to the category $\Sh(\widehat{G})$ as desired.

\end{proof}

We point out that by examining the proofs of Theorem \ref{modsheaf} and Theorem \ref{rodier}, we can obtain a simple description of the stalk of a representation $(\pi,V) \in \mathcal{R}(G)$ at a fixed $\chi \in \widehat{G}$. If we let $V_{\chi} \subset V$ be the $\C$-span of the vectors $\{ \pi(g)(v) - \chi(g) \cdot v \ | \ g \in G, v \in V \}$ (i.e. the $\chi$-eigenspace of $V$), then the stalk of $(\pi,V)$ at $\chi$ is identified with $V/V_{\chi}$, the space of $\chi$-coinvariants.

\end{document}